\documentclass[a4paper,11pt]{amsart}

\usepackage{epsfig}
\usepackage{amsthm,amsfonts}
\usepackage{amssymb,graphicx,color}
\usepackage{float}
\graphicspath{ {Images/} }
\usepackage[labelfont={bf},textfont={it}]{caption}
\usepackage[labelfont={bf, it}]{subcaption}
\usepackage[all]{xy}
\usepackage{verbatim}
\usepackage{hyperref}
\usepackage{enumitem}

\newtheorem{theorem}{Theorem}[section]
\newtheorem*{theorem*}{Theorem}
\newtheorem{lemma}[theorem]{Lemma}
\newtheorem{corollary}[theorem]{Corollary}
\newtheorem{proposition}[theorem]{Proposition}
\newtheorem{definition}[theorem]{Definition}

\newtheorem{question}{Question}
\newtheorem{claim}{Claim}
\newtheorem{example}[theorem]{Example}

\newcommand{\R}{\mathbb{R}}
\newcommand{\C}{\mathbb{C}}

\begin{document}

\title[H\"older invariance of the Henry-Parusinski invariant]
{H\"older invariance of the Henry-Parusinski invariant}

\author[A. Fernandes]{Alexandre Fernandes}
\author[J. E. Sampaio]{Jos\'e Edson Sampaio}
\address{Jos\'e Edson Sampaio and Alexandre Fernandes: 
Departamento de Matem\'atica, Universidade Federal do Cear\'a, Av. Humberto Monte, s/n Campus do Pici - Bloco 914, 60455-760
Fortaleza-CE, Brazil} 

 \email{alex@mat.ufc.br}

 \email{edsonsampaio@mat.ufc.br}

\author[J. P. da Silva]{Joserlan Perote da Silva}
\address{Joserlan Perote da Silva: Departamento de Matem\'atica, Universidade de Integra\c{c}\~ao Internacional da Lusofonia Afro-Brasileira (UNILAB), Campus dos Palmares, Cep. 62785-000. Acarape-Ce, Brasil} 

\email{joserlanperote@unilab.edu.br}

\thanks{The first named author was partially supported by CNPq-Brazil grant 304700/2021-5. The second named author was partially supported by CNPq-Brazil grant 310438/2021-7 and by the Serrapilheira Institute (grant number Serra -- R-2110-39576).}


\begin{abstract}
In this article, we show the
H\"older invariance of the Henry-Parusinski invariant. For a single germ $ f$, the Henry-Parusinski invariant of $ f $ is given in terms of the leading coefficients of the asymptotic expansion of $ f $ along the branches of the generic polar curve of $f$. As a consequence, we obtain that the classification problem of polynomial function-germs, with uniformly bounded degree, under H\"older equivalence, admits continuous moduli.
\end{abstract}

\maketitle

\section{Introduction}
 
In this paper, we study germs of analytic functions $f\colon(\mathbb{C}^2,0)\rightarrow(\mathbb{C},0)$ up to change of coordinates which are H\"older and their inverse are also H\"older. 

Let us recall the definition of an $\alpha$-H\"older mapping: for $\alpha\in (0,1]$, $X\subset\R^n$ and $Y\subset\R^m$, a mapping $f\colon X\rightarrow Y$ is called {\bf $\alpha$-H\"older} if there exists $\lambda >0$ such that
$$|f(x_1)-f(x_2)|\le \lambda |x_1-x_2|^{\alpha},\quad \mbox{for all } x_1,x_2\in X.$$ 
An $\alpha$-H\"older mapping $f\colon X\rightarrow Y$ is called
{\bf bi-$\alpha$-H\"older} if its inverse mapping exists and is $\alpha$-H\"older.

As it was already noted in \cite{FernandesFSS:2023}, there is nothing new in considering the possibility of classifying such function-germs up  to  change of coordinates that are less regular than analytic diffeomorphisms; 
Actually, for the family $f_r(x,y)=xy(x+y)(x-ry)$, $r\in \R$, Whitney proved that for any $t\not=s$, close enough to 0, there is no diffeomorphism of class $C^1$, $h\colon (\mathbb{C}^2,0)\to (\mathbb{C}^2,0)$, 
that maps the special fiber of $f_t $ onto the special fiber of $f_s$.
This shows that there is a family of polynomial function germs, with uniformly bounded degrees, that have uncountable different classes under changes of coordinates of class $C^1$. 
H. Whitney showed that the problem of classification of germs of analytic functions up to $C^1$ diffeomorphisms presents the so-called {\it continuous moduli}. Since then, some authors have investigated the possibility of classifying germs of functions under homeomorphisms. 
For example, Fukuda \cite{Fukuda:1976} showed that such a classification problem does not exhibit the continuous moduli phenomenon, as described above in Whitney's example. More precisely, Fukuda showed that for any family of polynomial function germs with uniformly bounded degrees, it has finitely many models under topological changes of coordinates. Many other authors analyzed singularities under changes of coordinates that were less regular than analytic and more regular than topological ones; for instance, Kuo \cite{Kuo:1985}. 
Recently, Henry and Parusinski \cite{HenryP:2003} considered germs of analytic functions in two complex variables under bi-Lipschitz changes of coordinates and, just as Whitney did for the case of analytic changes, showed the existence of continuous moduli for this problem. It is worthy to mention that, due to Mostowski's Finiteness Theorem \cite{Mostowski:1985}, unlike Whitney's example, Henry and Parusinski could not resorting only to the study of special fibers of such functions to show the existence of continuous moduli in this case and, therefore, the work of Henry and Parusinski impressed researchers in the area in a very positive way due to the novelty of the techniques that they developed.

Given this brief historical context, let us state exactly the classification problems that we are going to consider here; namely, the classification problem of analytic two complex variable function-germs up to some H\"older change of coordinates as described below: 

\begin{definition} Let $f,g\colon(\mathbb{R}^n,0)\rightarrow (\mathbb{R},0)$ be function-germs.
\begin{itemize}
 \item $f$ and $g$ are said {\bf bi-$\alpha$-H\"older conjugated} if there exists a bi-$\alpha$-H\"older homeomorphism $H\colon(\mathbb{R}^n,0)\rightarrow (\mathbb{R}^n,0)$ such that $f=g\circ H$;
 \item $f$ and $g$ are said {\bf H\"older equivalent} if $f$ and $g$ are bi-$\alpha$-H\"older conjugated for any $\alpha\in (0,1)$. In this case, we write $f\sim_{H} g$; 
 \item $f$ and $g$ are said {\bf Lipschitz equivalent} if  there exists a bi-1-H\"older homeomorphism $H\colon(\mathbb{R}^n,0)\rightarrow (\mathbb{R}^n,0)$ such that $f=g\circ H$. 
\end{itemize}
\end{definition}
It is worth saying that for a fixed $\alpha\in (0,1)$, bi-$\alpha$-H\"older conjugation does not define an equivalence relation since the composition of two bi-$\alpha$-H\"older homeomorphisms is not, in general, a bi-$\alpha$-H\"older homeomorphism. However, $\sim_{H}$ define an equivalence relation and we have the following sequence of implications:
$$
C^1 \, eq. \Rightarrow Lipschitz \, eq. \Rightarrow H\ddot{o}lder\, eq. \Rightarrow Top.\, eq.
$$

In the setting given by the above definition, we have the following question.

\begin{question}\label{main_question}
	Is there a family of polynomials $\{f_t\}_{t\in I}$ (where $I$ is an uncountable set) with uniformly bounded degree such that $f_t$ is not H\"older equivalent to $f_s$ for any $s\not=t$?
\end{question}

The main aim of this paper is to show in Theorem \ref{main theorem} that the Henry-Parunsinski invariant defined in \cite{HenryP:2003} is invariant under H\"older equivalence (see Subsection \ref{subsec:HP_inv} to see the definition of the Henry-Parunsinski invariant). As a consequence, we obtain that the classification problem of polynomial function-germs, with uniformly bounded degree, under H\"older equivalence, admits continuous moduli. In particular, we arrive at a positive answer to the Question \ref{main_question} and we get the H\"older equivalence of analytic function-germs $(\mathbb{C}^2,0)\rightarrow (\mathbb{C},0)$ admits continuous moduli. 
Actually, we obtain that: given an integer number $d>1$, the $1$-parameter family of weighted homogeneous $(\mathbb{C}^2,0)\rightarrow(\mathbb{C},0)$ given
\begin{equation*}
f_{t}\left( x,y\right) =x^{3}-3t^{2}xy^{2d}+y^{3d}
\end{equation*}
under the H\"older equivalence admits continuous moduli (see Subsection \ref{subsec:existence_moduli}). 

\bigskip


\section{Preliminaries}\label{sec:preliminaries}

All the subsets of $\R^n$ or $\mathbb{C}^n$ considered in the paper are supposed to be equipped with the Euclidean distance. When we consider other distance, it is clearly emphasized.

\subsection{Lipschitz and H\"older mappings.} In the following, we introduce the definitions of Lipschitz (bi-Lipschitz) and H\"older (bi-H\"older) mappings.

\begin{definition}
Let $X\subset \R^n$ and $Y\subset \R^m$ be two subsets and let $h\colon X\to Y$.
\begin{itemize}
 \item We say that $h$ is {\bf Lipschitz} (resp. {\bf $\alpha$-H\"older}) if there exists a positive constant $C$ such that $|h(x)-h(y)|\leq C|x-y|$ (resp. $|h(x)-h(y)|\leq C|x-y|^{\alpha}$) for all $ x, y\in X$.
 
\item We say that $h$ is {\bf bi-Lipschitz} (resp. {\bf bi-$\alpha$-H\"older}) if $h$ is a homeomorphism, Lipschitz (resp. $\alpha$-H\"older) and its inverse is also Lipschitz (resp. $\alpha$-H\"older). In this case, we also say that $X$ is bi-Lipschitz (resp. bi-$\alpha$-H\"older) homeomorphic to $Y$.
\end{itemize}
\end{definition}

\subsection{Other distances}
Given a path connected subset $X\subset\R^n$, the
{\bf inner distance}  on $X$  is defined as follows: given two points $x_1,x_2\in X$, $d_X(x_1,x_2)$  is the infimum of the lengths of paths on $X$ connecting $x_1$ to $x_2$.

Another important distance in this article is the {\bf diameter distance} on a subset $X\subset\R^N$, which is defined as follows: given two points $x_1,x_2\in X$, $d_{X,diam}(x_1,x_2)$  is the infimum of the diameters of the image of paths on $X$ connecting $x_1$ to $x_2$.

In general, one has that $|x_1-x_2|\leq d_{X,diam}(x_1,x_2)\leq d_{X}(x_1,x_2)$, for all pair of points $x_1,x_2\in X$.
The important class of sets where these distances are equivalent is defined in the following:

\begin{definition}\label{def:lne}
Let $X\subset\R^N$ be a subset. We say that $X$ is {\bf Lipschitz normally embedded (LNE)} if there exists a constant $C\geq 1$ such that $d_{X}(x_1,x_2)\leq C|x_1-x_2 |$, for all pair of points $x_1,x_2\in X$. In this case, we also say that $X$ is $C$-LNE. 
\end{definition} 

In general, these distances are not equivalent as we can see in the next examples that were already presented in \cite{Sampaio:2024}.

\begin{example}
Let $X=\{(x,y)\in \R^2; y^3=x^2\}$. Then, $d_{X}$ and $d_{X,diam}$ are equivalent, but it is well-known that $d_X$ and the distance induced by the Euclidean distance are not equivalent. 
\end{example}

\begin{example}
Let $G$ be the graph of the function $f\colon \R\to \R$ given by
$$
f(x)=\left\{\begin{array}{ll}
            x\sin (\frac{1}{x}),& \mbox{ if }x\not=0,\\
            0,& \mbox{ if }x=0.
            \end{array}\right.
$$
Let $X=\{(tx,ty,t)\in \R^3;t\geq 0$ and $(x,y)\in G\}$. For any point $(x,y)\in G\setminus\{(0,0)\}$, we have $d_X((x,y,1),(0,0,1))=1+|(x,y)|\geq 2$. However, $d_{X,diam}((x,y,1),(0,0,1))\to 0$ as $(x,y)\to (0,0)$. Thus, $d_{X,diam}$ is not equivalent to $d_X$.
\end{example}

However, we have the following result proved in \cite{Sampaio:2024}.
\begin{theorem}
 Let $X\subset\R^N$ be a path-connected subset. Assume that $X$ is a definable set in an o-minimal structure on $\R$. Then $d_{X, diam}$ is equivalent to $d_X$. 
\end{theorem}

In order to know more about the o-minimal geometry, see, for instance, \cite{Coste:1999}, \cite{Lojasiewicz:1964}, \cite{Gabrielov:1968} and \cite{BierstoneM:2000}.

\section{H\"older invariants}

\subsection{H\"older invariance of the multiplicity of functions}

\begin{proposition}\label{prop:mult_inv_conjugation}
Let $f,g:(\C^n,0)\to \C^p$ be two germs of analytic mappings. If $k={\rm ord}_0(f)\not={\rm ord}_0(g)=m$ then $f$ and $g$ are not bi-$\alpha$-H\"older conjugated for any $1>\alpha>\min\{\frac{k}{m},\frac{m}{k}\}$. 
\end{proposition}

Based on the results presented in \cite{ComteMT:2002} and  \cite{Sampaio:2019}, we have the following, which is more general than Proposition \ref{prop:mult_inv_conjugation}:
\begin{proposition}
Let $f,g:(\R^n,0)\to \R^p$ be two germs of analytic mappings. Assume $k={\rm ord}_0(f)\not={\rm ord}_0(g)=m$. Then there is no $\alpha\in (r, 1)$, where $r:=\min\{\frac{k}{m},\frac{m}{k}\}$, such that there are open neighborhoods $U, W\subset \R^n$ of $0\in \R^n$, constants $C_1,C_2>0$ and a bijection $\varphi\colon U\to W$  such that 
\begin{enumerate}
\item [{\rm (1)}] $\frac{1}{C_1}|x|^{\frac{1}{\alpha}}\leq |\varphi(x)|\leq C_1|x|^{\alpha}$, for all $x\in U$;
\item [{\rm (2)}] $\frac{1}{C_2}|f(x)|\leq |g\circ\varphi(x)|\leq C_2|f(x)|$, for all $x\in U$.
\end{enumerate} 
\end{proposition}
\begin{proof}
Assume that there is $\alpha\in (r, 1)$ such that there are open neighborhoods $U, W\subset \R^n$ of $0\in \R^n$, constants $C_1,C_2>0$ and a bijection $\varphi\colon U\to W$  such that 
\begin{enumerate}
\item [{\rm (1)}] $\frac{1}{C_1}|x|^{\frac{1}{\alpha}}\leq |\varphi(x)|\leq C_1|x|^{\alpha}$, for all $x\in U$;
\item [{\rm (2)}] $\frac{1}{C_2}|f(x)|\leq |g\circ\varphi(x)|\leq C_2|f(x)|$, for all $x\in U$.
\end{enumerate}
By changing $\varphi$ by $\varphi^{-1}$, if necessary, we may assume that $m<k$.
Since ${\rm ord}_0(f) =m$, there is some  $v\in \R^n\setminus \{0\}$ such that $L:=\lim \limits_{t\to 0^+}\frac{f(tv)}{t^m}\not =0$.
However, $t\mapsto \frac{g\circ \varphi (tv)}{t^{\alpha k}}$ is a bounded function around $0$. Since $k\alpha-m>0$, we obtain 
$$
\lim \limits_{t\to 0^+}\frac{|g\circ \varphi (tv)|}{t^{m}}=\lim \limits_{t\to 0^+}t^{k\alpha-m}\frac{|g\circ \varphi (tv)|}{t^{k\alpha}}=0.
$$
By using that
$$
\frac{1}{C_2}\frac{|f(tv)|}{t^m}\leq \frac{|g\circ \varphi (tv)|}{t^{\alpha k}}
$$
we obtain $L=0$, which is a contradiction.

Therefore, there is no such an $\alpha\in (r, 1)$.
\end{proof}

\begin{corollary}\label{cor:inv_mult}
Let $f,g\colon(\C^n,0)\to (\C,0)$ be two germs of analytic mappings. If $f$ and $g$ are H\"older equivalent, then ${\rm ord}_0(f)={\rm ord}_0(g)$.
\end{corollary}

\subsection{Special subsets preserved under bi-$\alpha$-H\"older conjugations}
Fix a single germ $g\colon(\mathbb{C}^{2},0)$ $\longrightarrow \left(\mathbb{C},0\right) \ $ and a point $p_{0}\in\mathbb{C}^{2}\ $ close to the origin, denote $c=g(p_{0})$. Let us fix a constant $K$ sufficiently large, which will be related to the $\alpha$-H\"{o}lder constant of a bi-$\alpha $-H\"{o}lder homeomorphism. Let $B(p_{0},\rho)$ denote the open ball centered at $p_{0}$ and of radius $\rho$. Denote
\begin{equation*}
X\left( p_{0},\rho \right) :=B\left( p_{0},\rho \right) \cap g^{-1}\left(
c\right).
\end{equation*}
Suppose that $p,q\in $ $X\left( p_{0},\rho \right) \ $ belong to the same connected component of $X(p_{0},K\rho ^{\alpha ^{2}})$ as $\ p_{0}$.
Let $d_{p_{0},\rho ,K,diam}(p,q)\ $ denote the diameter distance of $X(p_{0},K\rho ^{\alpha ^{2}})$ between $p$ and $q$.

Define
\begin{equation*}
\varphi _{i,\alpha}\left( p_{0},\rho ,K\right) :=\sup \dfrac{\left(
d_{p_{0},\rho ,K,diam}\left( p,q\right) \right) ^{1/\alpha ^{2-i}}}{\left\vert
p-q\right\vert ^{\alpha ^{2-i}}},\ i=0,1,2,
\end{equation*}%
where the supremum is taken over all pairs of points $p,q\in 
X(p_{0},\rho )\ $ in the connected component of $X(p_{0},K\rho ^{\alpha
^{2}}) $ that contains $p_{0}$.  The function $\varphi _{i,\alpha}$ ($i=0,1,2$) is not necessarily an increasing function
of $\rho $, so we define
\begin{equation*}
\psi _{i,\alpha}\left( p_{0},\rho ,K\right) :=\sup_{\rho'\leq \rho
}\varphi _{i,\alpha}\left( p_{0},\rho',K\right) ,\ i=0,1,2.
\end{equation*}
Finally, we define
\begin{equation*}
Y_{i,\alpha}(\rho ,K,A):=\{p;\psi _{i,\alpha}\left( p,\rho ,K\right) \geqslant A\},\
i=0,1,2.
\end{equation*}%
Intuitively speaking, $Y_{i,\alpha}(\rho ,K,A)$, for $A$ large, is the set of points where the curvature
of the levels $g^{-1}\left( c\right) $ is very large. We shall show that such $Y_{i,\alpha}(\rho ,K,A)$ are preserved by bi-$\alpha $-H\"{o}lder homeomorphisms.

Let $H\colon(\mathbb{C}^{2},0)\to (\mathbb{C}^{2},0)$ be the germ of a bi-$\alpha$-H\"{o}lder homeomorphism such that $H$ sends the levels of $\tilde{g}$ to the levels of $\ g$, where $g,\tilde{g}\colon(\mathbb{C}^{2},0)\to (\mathbb{C},0)$ are the germs of analytic functions. Fix $\ L\geqslant
1 $, a common $\alpha $-H\"{o}lder constant of $H$ and its inverse  $H^{-1}$. For $p_{0}\in\mathbb{C}^{2}$, we denote by $\tilde{p}_{0}=H\left( p_{0}\right)$, and similarly we add the tilde to distinguish the corresponding objects in the domain and the target space of $H$, that is, for instance,
\begin{equation*}
\tilde{Y}_{i,\alpha}(\rho ,K,A):=\{\tilde{p};\tilde\psi _{i,\alpha}\left( \tilde{p},\rho
,K\right) \geqslant A\},i=0,1,2,
\end{equation*}%
denotes a subset of the target of $H$.
\vspace{0,3cm}
\begin{lemma}\label{lemma0} 
If $K\geqslant L^{1+\tfrac{1}{\alpha }}$, then
\begin{equation*}
\tilde{Y}_{0,\alpha}(L^{-\tfrac{1}{\alpha }}\rho ^{\tfrac{1}{\alpha }},K,AL^{\alpha
+\tfrac{1}{\alpha ^{2}}})\subset H\left( Y_{1,\alpha}\left( \rho ,K,A\right)
\right) \subset \tilde{Y}_{2,\alpha}(L\rho ^{\alpha },K,AL^{-1-\tfrac{1}{\alpha }}).
\end{equation*}%
\end{lemma}
\begin{proof}
Since $H$ is bi-$\alpha $-H\"{o}lder, we have%
\begin{equation*}
\tilde{X}(\tilde{p}_{0},L^{-\tfrac{1}{\alpha }}\rho ^{\tfrac{1}{\alpha }%
})\subset H\left( X\left( p_{0},\rho \right) \right) \subset \tilde{X}\left(
\tilde{p}_{0},L\rho ^{\alpha }\right).
\end{equation*}%
Hence, if $K\geqslant L^{1+\tfrac{1}{\alpha }},$%
\begin{equation}\label{eq2.7}
\tilde{X}(\tilde{p}_{0},L^{-\tfrac{1}{\alpha }}\rho ^{\tfrac{1}{\alpha }%
})\subset H\left( X\left( p_{0},\rho \right) \right) \subset \tilde{X}(%
\tilde{p}_{0},kL^{-\tfrac{1}{\alpha }}\rho ^{\alpha })\subset
H(X(p_{0},K\rho ^{\alpha ^{2}})).
\end{equation}%
If $\tilde{p}_{0},\tilde{p}$ $\in \tilde{X}(\tilde{p}_{0},L^{-\tfrac{1}{%
\alpha }}\rho ^{\tfrac{1}{\alpha }})\ $ are in the same connected component of $\tilde{X}(\tilde{p}_{0},kL^{-\tfrac{1}{\alpha }}\rho ^{\alpha })$
then $p_{0}\  \mbox{and} \ p=H^{-1}\left( \tilde{p}\right) \ $ are in the same connected component of $X(p_{0},K\rho ^{\alpha ^{2}})$. Hence, (\ref{eq2.7}) implies
\begin{equation*}
\psi _{1,\alpha}\left( p_{0},\rho ,K\right) \geq L^{-\alpha -\tfrac{1}{\alpha ^{2}}%
}\tilde\psi _{0,\alpha}(\tilde{p}_{0},L^{-\tfrac{1}{\alpha }}\rho ^{\tfrac{1}{\alpha }},K).
\end{equation*}%
In fact, given\ $\tilde{p}_{1},\tilde{p}_{2}\in \tilde{X}(\tilde{p}_{0},kL^{-%
\tfrac{1}{\alpha }}\rho ^{\alpha })$ in the same connected component that contains $\tilde{p}_{0}.$ Let $p_{1}=H^{-1}\left( \tilde{p}_{1}\right) ,p_{2}=H^{-1}\left(\tilde{p}_{2}\right) $ and $\gamma \colon\left[ a,b\right] \rightarrow
X(p_{0},K\rho ^{\alpha ^{2}})$ given by $\gamma \left( a\right) =p_{1}\mbox{ and }\gamma \left( b\right) =p_{2}$ joining $p_{1}\mbox{ and } p_{2}.$ Let $\beta =H\circ
\gamma \colon\left[ a.b\right] \rightarrow \tilde{X}(\tilde{p}_{0},kL^{-\tfrac{1}{%
\alpha }}\rho ^{\alpha })$ given by $\beta \left( a\right) =H\left( \gamma
\left( a\right) \right) =H\left( p_{1}\right) =\tilde{p}_{1}\mbox{ and } \beta
\left( b\right) =H\left( \gamma \left( b\right) \right) =H\left(
p_{2}\right) =\tilde{p}_{2}$ joining $\tilde{p}_{1}\mbox{ and } \tilde{p}_{2}$.
Hence,
\begin{eqnarray*}
d_{\tilde{p}_{0},L^{-1/\alpha }\rho ^{1/\alpha },K,diam}\left( \tilde{p}_{1},%
\tilde{p}_{2}\right) &=&\inf_{\mu }diam\left( \mu \right) \leq
diam\left( \beta \right) =\sup \left\vert \beta \left( t_{i}\right) -\beta
\left( t_{j}\right) \right\vert \\
&=&\sup \left\vert H\left( \gamma \left( t_{i}\right) \right) -H\left(
\gamma \left( t_{j}\right) \right) \right\vert \\
&\leq &L\sup \left\vert \gamma \left( t_{i}\right) -\gamma \left(
t_{j}\right) \right\vert ^{\alpha }=L\left( diam\left( \gamma \right)
\right) ^{\alpha }.
\end{eqnarray*}%
So, \begin{equation}\label{eq2.6} 
d_{\tilde{p}_{0},L^{-1/\alpha }\rho ^{1/\alpha },K,diam}\left(
\tilde{p}_{1},\tilde{p}_{2}\right) \leq L\left( d_{\tilde{p}_{0},\rho
,K,diam}\left( p_{1},p_{2}\right) \right) ^{\alpha }.
\end{equation}
Thus,
\begin{equation*}
(d_{\tilde{p}_{0},L^{-1/\alpha }\rho ^{1/\alpha },K,diam}\left( \tilde{p}_{1},
\tilde{p}_{2}\right) )^{\tfrac{1}{\alpha ^{2}}}\leq L^{\tfrac{1}{\alpha ^{2}}
}\left( d_{\tilde{p}_{0},\rho ,K,diam}\left( p_{1},p_{2}\right) \right) ^{
\tfrac{1}{\alpha }}.
\end{equation*}
See also that
\begin{eqnarray*}
\left\vert p_{1}-p_{2}\right\vert &\leq &L\left\vert \tilde{p}_{1}-\tilde{p}
_{2}\right\vert ^{\alpha }\Rightarrow \left\vert p_{1}-p_{2}\right\vert
^{\alpha }\leq L^{\alpha }\left\vert \tilde{p}_{1}-\tilde{p}_{2}\right\vert
^{\alpha ^{2}} \\
&\Rightarrow &\dfrac{1}{\left\vert \tilde{p}_{1}-\tilde{p}_{2}\right\vert
^{\alpha ^{2}}}\leq \dfrac{L^{\alpha }}{\left\vert p_{1}-p_{2}\right\vert
^{\alpha }}.
\end{eqnarray*}
Then, $\dfrac{(d_{\tilde{p}_{0},L^{-1/\alpha }\rho ^{1/\alpha
},K,diam}\left( \tilde{p}_{1},\tilde{p}_{2}\right) )^{\tfrac{1}{\alpha ^{2}}}}{%
\left\vert \tilde{p}_{1}-\tilde{p}_{2}\right\vert ^{\alpha ^{2}}}$ $\leq
L^{\alpha +\tfrac{1}{\alpha ^{2}}}\dfrac{\left( d_{\tilde{p}_{0},\rho
,K,diam}\left( p_{1},p_{2}\right) \right) ^{\tfrac{1}{\alpha }}}{\left\vert
p_{1}-p_{2}\right\vert ^{\alpha }}$.\\
\vspace{0,1cm}\\
Hence,
\begin{equation*}
L^{-\alpha -\tfrac{1}{\alpha ^{2}}}\tilde\psi _{0,\alpha}(\tilde{p}_{0},L^{-\tfrac{1}{%
\alpha }}\rho ^{\tfrac{1}{\alpha }},K)\leq \psi _{1,\alpha}\left( p_{0},\rho
,K\right) .
\end{equation*}%

Similarly, we obtain the following
\begin{equation*}
\psi _{1,\alpha}\left( p_{0},\rho ,K\right) \leq L^{1+\tfrac{1}{\alpha }}\tilde\psi
_{2,\alpha}\left( \tilde{p}_{0},L\rho ^{\alpha },K\right) .
\end{equation*}%
\end{proof}

\begin{lemma}\label{lemma-semcite} 
If $K\geqslant L^{1+\tfrac{1}{\alpha }}$, then
\begin{eqnarray*}
\tilde{Y}_{0,\alpha}(L^{-\tfrac{1}{\alpha }}\rho ^{\tfrac{1}{\alpha }},K,AL^{\alpha
+\tfrac{1}{\alpha ^{2}}})&\subset&  \tilde{Y}_{1,\alpha}\left( L\rho^{\alpha} ,K,AL^{1-\tfrac{1}{\alpha^{2}}}\right) \\
&\subset& \tilde{Y}_{2,\alpha}(L^{1+\alpha}\rho ^{\alpha^{3}},K,AL^{-\tfrac{2}{\alpha}-\tfrac{1}{\alpha^{2} }-\tfrac{1}{\alpha^{3}}}).
\end{eqnarray*}%
\end{lemma}
\begin{proof}
From Lemma \ref{lemma0}, we have $\tilde{Y}_{0,\alpha} \subset \tilde{Y}_{2,\alpha}$. Now, let us show that  $\tilde{Y}_{0,\alpha} \subset \tilde{Y}_{1,\alpha}$. More precisely, $\tilde{Y}_{0,\alpha}(L^{-\tfrac{1}{\alpha }}\rho ^{\tfrac{1}{\alpha }},K,AL^{\alpha
+\tfrac{1}{\alpha ^{2}}})\subset  \tilde{Y}_{1,\alpha}\left( L\rho^{\alpha} ,K,AL^{1-\tfrac{1}{\alpha^{2}}}\right)$.
In a similar way to the Inequality (\ref{eq2.6}), we have 
$$
d_{{p}_{0},\rho,K,diam}\left({p}_{1},{p}_{2}\right) \leq L\left( d_{\tilde{p}_{0},L\rho^{\alpha},K,diam}\left( \tilde{p}_{1},\tilde{p}_{2}\right) \right) ^{\alpha }.
$$
Thus,
\begin{eqnarray*}
(d_{\tilde{p}_{0},L^{-1/\alpha }\rho ^{1/\alpha },K,diam}\left( \tilde{p}_{1},%
\tilde{p}_{2}\right) )^{\tfrac{1}{\alpha ^{2}}}&\leq& L^{\tfrac{1}{\alpha ^{2}}%
}\left( d_{{p}_{0},\rho ,K,diam}\left( p_{1},p_{2}\right) \right) ^{%
\tfrac{1}{\alpha }}\\
&\leq& L^{\tfrac{1}{\alpha}+\tfrac{1}{\alpha ^{2}}%
} d_{\tilde{p}_{0},L \rho^{\alpha} ,K,diam}\left( \tilde{p}_{1},\tilde{p}_{2}\right).
\end{eqnarray*}%
Then, 
\begin{eqnarray*}
\dfrac{(d_{\tilde{p}_{0},L^{-1/\alpha }\rho ^{1/\alpha
},K,diam}\left( \tilde{p}_{1},\tilde{p}_{2}\right) )^{\tfrac{1}{\alpha ^{2}}}}{%
\left\vert \tilde{p}_{1}-\tilde{p}_{2}\right\vert ^{\alpha ^{2}}} &\leq&
L^{\tfrac{1}{\alpha} +\tfrac{1}{\alpha ^{2}}}\dfrac{ d_{\tilde{p}_{0},L\rho^{\alpha},K,diam}\left( \tilde{p}_{1},\tilde{p}_{2}\right)}{\left\vert
\tilde{p}_{1}-\tilde{p}_{2}\right\vert ^{\alpha^{2} }}\\
&\leq&
L^{\tfrac{1}{\alpha} +\tfrac{1}{\alpha ^{2}}}\left( \dfrac{\left( d_{\tilde{p}_{0},L\rho^{\alpha},K,diam}\left( \tilde{p}_{1},\tilde{p}_{2}\right) \right) ^{\tfrac{1}{\alpha }}} {\left\vert
\tilde{p}_{1}-\tilde{p}_{2}\right\vert ^{\alpha }}\right)^{\alpha}.
\end{eqnarray*}
Hence,
\begin{equation*}
L^{-\tfrac{1}{\alpha} -\tfrac{1}{\alpha ^{2}}}\tilde\psi _{0,\alpha}(\tilde{p}_{0},L^{-\tfrac{1}{%
\alpha }}\rho ^{\tfrac{1}{\alpha }},K)\leq \left( \tilde{\psi} _{1,\alpha}\left( \tilde{p}_{0},L\rho^{2}
,K\right)\right) ^{\alpha} .
\end{equation*}%
$\ $ 
Since $L^{\alpha+\tfrac{1}{\alpha ^{2}}}A < \tilde\psi _{0,\alpha}(\tilde{p}_{0},L^{-\tfrac{1}{%
\alpha }}\rho ^{\tfrac{1}{\alpha }},K)$, we have
\begin{equation*}
L^{-\tfrac{1}{\alpha} -\tfrac{1}{\alpha ^{2}}}L^{\alpha+\tfrac{1}{\alpha ^{2}}}A <  \left( \tilde{\psi} _{1,\alpha}\left( \tilde{p}_{0},L\rho^{2}
,K\right)\right) ^{\alpha} .
\end{equation*}%
Therefore,
\begin{equation*}
L^{1-\tfrac{1}{\alpha^{2} }}A  <  \left( L^{\alpha-\tfrac{1}{\alpha}}\right) ^{\tfrac{1}{\alpha}}{A}^{\tfrac{1}{\alpha}}  < \tilde\psi_{1,\alpha}\left( \tilde{p}_{0},L\rho ^{\alpha },K\right) .
\end{equation*}%
\end{proof}

A similar argument shows the following:

\begin{lemma}\label{lema2.2}
Let $\delta >0\ $ and denote 
$$
Y_{i,\alpha}(\delta ,M,K,A):=\{p;\psi
_{i,\alpha}(p,M^{\tfrac{1}{\alpha ^{8-2i}}}\left\vert p\right\vert ^{\tfrac{%
1+\delta }{\alpha ^{8-2i}}},K)\geqslant A\},\ i=0,1,2.
$$
If $\ K\geqslant L^{1+%
\tfrac{1}{\alpha }+2\tfrac{1+\delta }{\alpha ^{5}}},\ $ then
\begin{eqnarray*}
i) \ \ \tilde{Y}_{0,\alpha}(\delta ,L^{-\tfrac{1}{\alpha }-\tfrac{1+\delta }{\alpha ^{8}}%
}M^{\tfrac{1}{\alpha ^{8}}},K,AL^{\alpha +\tfrac{1}{\alpha ^{2}}}) &\subset
&H(Y_{1,\alpha}(\delta ,M^{\tfrac{1}{\alpha ^{6}}},K,A)) \\
&\subset &\tilde{Y}_{2,\alpha}(\delta ,L^{1+\tfrac{1+\delta }{\alpha ^{5}}}M^{%
\tfrac{1}{\alpha ^{4}}},K,AL^{-1-\tfrac{1}{\alpha }}).
\end{eqnarray*}%
\begin{eqnarray*}
ii) \ \ \tilde{Y}_{0,\alpha}(\delta ,L^{-\tfrac{1}{\alpha }-\tfrac{1+\delta }{\alpha ^{8}}%
}M^{\tfrac{1}{\alpha ^{8}}},K,AL^{\alpha +\tfrac{1}{\alpha ^{2}}}) &\subset &
H(Y_{1,\alpha}(\delta ,M^{\tfrac{1}{\alpha ^{6}}},K,A)) \\
 &\subset & \tilde{Y}_{2,\alpha}(\delta ,L^{1+\tfrac{1+\delta }{\alpha ^{5}}}M^{%
\tfrac{1}{\alpha ^{4}}},K,AL^{-1-\tfrac{1}{\alpha }}).
\end{eqnarray*}%
\end{lemma}

\subsection{The Henry-Parunsinski Invariant}\label{subsec:HP_inv} 

In order to distinguish bi-Lipschitz types of complex analytic function germs of two complex variables $f \colon (\C^2, 0) \to (\C, 0)$, it was constructed in \cite{HenryP:2003} a numerical invariant that is
given in terms of the leading coefficients of the asymptotic expansions of $f$ along the
branches of generic polar curve of $f$. 

We recall the main result of \cite{HenryP:2003}. Let $f \colon (\C^2, 0) \to (\C, 0)$ be the germ of an analytic function with Taylor expansion:
\begin{equation}\label{eq1} 
 f(x, y) = H_k(x, y) + H_{k+1}(x, y) +\cdots
\end{equation}
where $H_k\not=0$ and each $H_j$ is a homogeneous polynomial that has degree $j$ or $H_j\equiv 0$. We shall assume that $H_k(1, 0) \not= 0$, in this case, we say that $f$ is {\bf mini-regular in $x$}. We also assume, for simplicity, that $f(x, y)$ has no
multiple roots.
By an analytic arc we mean a fractional power series of the form
\begin{equation}\label{eq2} 
\lambda : x=\lambda (y):= c_{1}y^{n_{1}/N}+c_{2}y^{n_{2}/N}+ \cdots, c_{i} \in \mathbb{C}
\end{equation}
where $N \leq n_1 < n_2 < \cdots $ are positive integers having no common divisor, such that
$\lambda (t^{N})$ has positive radius of convergence. We can identify $\lambda $ with the analytic arc $x=\lambda (y):=c_{1}t^{n_{1}}+c_{2}t^{n_{2}}+ \cdots,$ $\vert t \vert$ small, which is not tangent to the $x$-axis (since $n_1/N \geq 1$).
A polar arc $x=\gamma (y)$ is a branch of the polar curve $\Gamma:\partial f/\partial x =0.$ Since $f$ is
mini-regular in $x,$ $x=\gamma (y)$ is not tangent to the $x$-axis and it is an arc in our sense. Let
$\gamma$ be a polar arc. We associate to $\gamma$ two numbers: $h_0 = h_0(\gamma) \in \mathbb{Q_{+}}$ and $c_{0} = c_{0}(\gamma) \in \mathbb{C^{*}}$ given by the expansion
\begin{equation}\label{eq4} 
f(\gamma(y),y) = c_{0}y^{h_{0}} + \cdots, \quad c_{0} \neq 0.
\end{equation}
If $h_0(\gamma) > k$ then the polar arc has to be tangent to the singular locus of the tangent cone
$C_0(X)$, given by $Sing(C_{0}(X)) := \{\partial H_{k} / \partial x = \partial H_{k} / \partial y = 0 \}$, where $X=f^{-1}(0)$. We call such polar arcs as {\bf tangential polar arcs}.
Fix a line $l \subset Sing(C_{0}(X)).$ Let $\Gamma(l)$ denote the set of polar arcs tangent to $l$.
We associate to $l$ the set of formal expressions
$I(l) = \{c_{0}(\gamma)y^{h_{0}(\gamma)} | \gamma \in \Gamma(l) \} / \mathbb{C^{*}},$
where $c \in \mathbb{C^{*}}$ acts by multiplication on $y$:
$$
\{c_{01}y^{h_{01}},\cdots, c_{0k}y^{h_{0k}} \} \sim \{c_{01}c^{h_{01}}y^{h_{01}},\cdots,c_{0k}c^{h_{0k}}y^{h_{0k}} \}.
$$
By the invariant ${\rm Inv}(f)$ of $f$ we mean the set of all $I(l)$, where $l$ runs over all lines in
$Sing(C_{0}(X)).$ No that ${\rm Inv}(f)$ is well-defined and does not depend on the choice of local coordinates. The main result of \cite{HenryP:2003} is the following.

\begin{theorem}
Let $f_1, f_2\colon (\C^2, 0)\to (\C, 0)$ be two analytic function-germs. If $f_1$ and $f_2$ are Lipschitz equivalent, then ${\rm Inv}(f_1) = {\rm Inv}(f_2)$. 
\end{theorem}

Let us introduce the main goal of this paper, namely, we state the H\"older invariance of the Henry-Parusinski number ${\rm Inv}(f)$ for function-germs $f\colon(\C^2,0)\to (\C^2,0)$.

\begin{theorem}\label{main theorem} 
Let $f_1,f_2\colon(\C^2,0)\to (\C,0)$ be two germs of analytic functions. If $f_1$ and $f_2$  are H\"older equivalent, then ${\rm Inv}(f_1)={\rm Inv}(f_2)$.
\end{theorem}

Before we show Theorem \ref{main theorem} that deals with the Henry-Parunsinski Invariant, we observe two important results about the sets $Y_{\delta,i,\alpha}:=Y_{i,\alpha}(\delta ,M,K,A)$,  $i=0,1,2$. Firstly, we show that they are preserved by bi-$\alpha$-H\"older homeomorphism, see Lemma \ref{lema2.2}. Secondly, we show that the asymptotic behavior of a function $f$ on $Y_{\delta,i,\alpha}$ ($i=0,1,2$) gives exactly the invariant ${\rm Inv}(f)$ thanks to the following proposition, which will be proved in the next section.

\begin{proposition}\label{proposition main} 
Let $\tilde{\Gamma}$ be the union of tangential polar arcs. Suppose $\delta > 0$ and sufficiently small. Then \\
$ a) \ \tilde{\Gamma} \subset Y_{\delta,i,\alpha}$ $\rm ($$i=0,1,2$$\rm )$.\\ 
$ b)$ There is a constant $B > 0,$ which depends on the constants in the definition of $Y_{\delta,i,\alpha}$  $\rm ($$i=0,1,2$$\rm )$ such that for each $p_0 \in Y_{\delta,,i,\alpha}$ $\rm ($$i=0,1,2$$\rm )$ there is $p \in \tilde{\Gamma}$ such that $f(p) = f(p_0)$ and
\begin{equation}
| p - p_{0} | \leq B {| p_{0} |}^{1+ \delta}
\end{equation}
\end{proposition}

\

\begin{proof}[ Proof of Theorem \ref{main theorem}]

Let $X$ (resp. $Y$) be the zeros of $f_1$ (resp. $f_2$).

We first prove the following:

\begin{claim}\label{claim:inv_reduced}
 ${\rm Inv}(f_1)=\emptyset$ if and only if $C_0(X)$ is the union of $k$ different lines, where $k={\rm ord}_0 (f_1)$.
\end{claim}
\begin{proof}[Proof of Claim \ref{claim:inv_reduced}]
 It follows from the definition of 
 ${\rm Inv}(f_1)$ that ${\rm Inv}(f_1)=\emptyset$ if and only if there is no line $l \subset Sing(C_{0}(X)).$ However, there is no line $l \subset Sing(C_{0}(X))$ if and only if the initial part of $f_1$, $H_k$, is a reduced polynomial. 
Since $C_0(X)$ is the zeros set of $H_k$, we obtain that $H_k$ is a reduced polynomial if and only if $C_0(X)$ is the union of $k$ different lines.

Therefore, ${\rm Inv}(f_1)=\emptyset$ if and only if $C_0(X)$ is the union of $k$ different lines.
\end{proof}

Assume that $f_1$ and $f_2$ are H\"older equivalent. From Corollary \ref{cor:inv_mult}, ${\rm ord}_0(f_1)={\rm ord}_0(f_2)=k$.

\begin{claim}\label{claim:inv_empty_equiv}
 ${\rm Inv}(f_1)=\emptyset$ if and only if ${\rm Inv}(f_2)=\emptyset$.
\end{claim}
\begin{proof}[Proof of Claim \ref{claim:inv_empty_equiv}]
Assume that ${\rm Inv}(f_1)=\emptyset$. By Claim \ref{claim:inv_reduced}, $C_0(X)$ is the union of $k$ different lines.

In order to show that ${\rm Inv}(f_2)=\emptyset$, by Claim \ref{claim:inv_reduced}, it is enough to show that $C_0(Y)$ is the union of $k$ different lines.

Since $f_1$ and $f_2$ are H\"older equivalent, then for any $\alpha\in (0,1)$ there is a bi-$\alpha$-H\"older homeomorphism $\varphi_{\alpha}\colon (\C^2,0)\to (\C^2,0)$ such that $\varphi_{\alpha}(X)=Y$. By \cite[Corollary 4.9]{FernandesSS:2018}, there is a bi-Lipschitz homeomorphism $\phi\colon (X,0)\to (Y,0)$. By \cite[Theorem 3.2]{Sampaio:2016}, there is a bi-Lipschitz homeomorphism $\psi\colon (C_0(X),0)\to (C_0(Y),0)$. Therefore, $C_0(Y)$ is the union of $k$ different lines as well. By Claim \ref{claim:inv_reduced}, ${\rm Inv}(f_2)=\emptyset$.

\end{proof}

Finally, let us consider the case where ${\rm Inv}(f_1)$ and ${\rm Inv}(f_2)$ are not empty sets; we are going to show that  ${\rm Inv}(f_1)= {\rm Inv}(f_2)$. First, we explain how to recover ${\rm Inv}(f_1)$ from the asymptotic behavior of $f_1$ on $Y_{\delta}.$ By Proposition \ref{proposition main}, the tangent cone to $Y_{\delta}$ at the origin coincides with $Sing(C_{0}(X))$, and then, as germ at the origin, $Y_{\delta}$ is included in a ``horn" neighborhood of $Sing(C_{0}(X))$, that is,
$$Y_{\delta}\subset\{ p \in \mathbb{C}^{2}; dist(p,Sing(C_{0}(X))) \leq {| p |}^{1+\xi}\}$$ for some $\xi > 0,$ where $dist(p,A)=\inf \{|p-a|;a\in A\}$. Indeed, clearly the union $\tilde{\Gamma}$ of tangential polar arcs is included in such a ``horn" neighborhood for an exponent $\xi.$ By taking $\xi$ even smaller (if necessary), in particular smaller than $\delta$, we may assure, by Proposition \ref{proposition main}, that the ``horn" neighborhood contains $Y_{\delta}.$ Given $l \subset Sing(C_{0}(X))$, denote by $Y_{\delta}(l)$ the part of $Y_{\delta}$ tangent to $l$ that is 
$$Y_{\delta} \cap \{p \in \mathbb{C}^{2}; dist( p,Sing(C_{0}(X)))\leq {| p |}^{1+ \xi } \}.$$ Given $p_{0} = (x_{0}, y_{0}) \in Y_{\delta}(l)$ and $C > 0$, let us denote
\begin{equation}\label{def1} 
V(p_{0}) = \{ f_1(p); p \in Y_{\delta} \ \mbox{and} \ | p - p_{0} | \leq C{| p_{0} |}^{1+ \delta} \}.
\end{equation}

\begin{lemma}\label{lemma1} 
There is $\eta > 0$ such that \begin{equation*}
V(p_{0})\subset \bigcup\limits_{\gamma \in \Gamma(l)} \{\tau\in\C ; | \tau - c_{0}({\gamma})y_{0}^{h_0({\gamma})} |  \leq {| y_{0} | }^{h_0({\gamma})+ \eta}  \}.
\end{equation*} 
\end{lemma}
\begin{proof}
Let $p \in Y_{\delta}$ and let $| p - p_0 | \leq B{| p_0| }^{1+\delta}.$ By Proposition \ref{proposition main}, we may suppose that $p_0$ and $p$ belong to $\Gamma(l)$ without changing the values of $f_1$ on them; which proves the lemma.
\end{proof}

Let $f_{1},$ $f_{2}$ be two analytic functions germs such that $f_{1} = f_{2} \circ H,$ where $H$ is a bi-$\alpha$-H\"older homeomorphism. By Lemma \ref{lema2.2}, $H(Y_{\delta,1,\alpha,f_{1}}) \subset Y_{\delta,2,\alpha,f_{2}}$. Fix $l_{1} \subset Sing(C_{0}( f_{1}^{-1}(0)))$ and a polar arc $\gamma$ tangent to $l_{1}.$ The image $H(\gamma)$ belongs to $Y_{\delta,2,\alpha, f_{2}}$, hence it belongs to $Y_{\delta,2,\alpha,f_{2}}(l_{2})$, for some $l_{2} \subset Sing(C_{0}( f_{2}^{-1}(0))).$ Let us point out that the argument above to get the line $l_2$ does not depend on the choice of $\gamma$ with tangent line $l_1$. Indeed, all such polar arcs $\gamma$ are mutually tangent and $H$ is an $\alpha$-H\"older mapping, by Proposition \ref{proposition main}, it implies that
$H(Y_{\delta,1,\alpha,f_{1}}(l_{1})) \subset Y_{\delta,2,\alpha,f_{2}}(l_{2}).$ 

Let $p_{1} \in Y_{\delta,1,\alpha,f_{1}}(l_{1})$ and let $p_{2} = (x_{2}, y_{2}) = H(p_{1}).$ By the above arguments,
\begin{equation}\label{eq3} 
V(p_{1}) \subset V(H(p_{1})),
\end{equation}
maybe for a different constant $C$ in the definition (\ref{def1}). Denote by ${\Gamma}_{i}(l_{i})$ the union of the all polar arcs in $Y_{\delta,i,\alpha,f_{i}}(l_{i}),$ for $i = 1,2$, respectively. By Lemma \ref{lemma1} and the inclusion in (\ref{eq3}), we have $$V(p_{1}) \subset \bigcup\limits_{\tilde{\gamma} \in {\Gamma}_{2}(l_{2})}  \{ \tau\in\C ; | \tau - c_{0}(\tilde{\gamma})y_{2}^{h_{0}(\tilde{\gamma})} | \leq {| y_{2} | }^{h_{0}(\tilde{\gamma})+ \eta} \}.$$ 
Thus, for each polar arc $x_{1} = \gamma(y_{1})$ in $\Gamma_{1}(l_{1}),$ we have $f_{1}(\gamma(y_{1}), y_{1}) \in V(p_{1})$ and $f_{1}(\gamma(y_{1}),y_{1}) = c_{0}(\gamma)y_{1}^{h_{0}(\gamma)}+ o(y_{1}^{h_{0}(\gamma)+ \eta})$, and therefore
$$
f_{1}(\gamma(y_{1}),y_{1}) \in  \bigcup\limits_{\tilde{\gamma} \in {\Gamma}_{2}(l_{2})}  \{ \tau ; | \tau - c_{0}(\tilde{\gamma})y_{2}^{h_{0}(\tilde{\gamma})} | \leq {| y_{2} | }^{h_{0}(\tilde{\gamma})+ \eta} \}.
$$
As a consequence of that, for each $\gamma$ in $\Gamma_{1}(l_{1})$ there is a $\tilde{\gamma} \in \Gamma_{2}(l_{2})$ such that
\begin{equation}\label{Eq:final}
\vert c_{0}(\gamma)y_{1}^{h_0(\gamma)} - c_{0}(\tilde{\gamma})y_{2}^{h_0(\tilde{\gamma})}\vert \leq {\vert y_2\vert }^{h_0(\tilde{\gamma})+ \eta}.
\end{equation}

We claim that $h_0(\gamma)=h_0(\tilde{\gamma}).$ In fact, let us suppose $h_0(\gamma)\neq h_0(\tilde{\gamma})$ and consider $\alpha > min\{\tfrac{h_0(\gamma)}{h_0(\tilde{\gamma})},\tfrac{h_0(\tilde{\gamma})}{h_0(\gamma)}\}.$ Without loss of generality, one can suppose $h_0(\gamma) > h_0(\tilde{\gamma})$; in this case, we have
$\alpha > \tfrac{h_0(\tilde{\gamma})}{h_0(\gamma)}$, which implies that $\alpha h_0(\gamma)>h_0(\tilde{\gamma}).$ Then
\begin{eqnarray*}
\vert c_{0}(\tilde{\gamma}) \vert {\vert y_{2}\vert}^{h_0(\tilde{\gamma})} - \vert c_{0}(\gamma)\vert {\vert y_{1} \vert}^{h_0(\gamma)} \leq \vert c_{0}(\gamma)y_{1}^{h_0(\gamma)} - c_{0}(\tilde{\gamma})y_{2}^{h_0(\tilde{\gamma})} \vert \leq {\vert y_2\vert }^{h_0(\tilde{\gamma})+ \eta}.
\end{eqnarray*}
Therefore
\begin{eqnarray*}
\vert c_{0}(\tilde{\gamma}) \vert {\vert y_{2}\vert}^{h_0(\tilde{\gamma})} \leq \vert c_{0}(\gamma)\vert {\vert y_{1} \vert}^{h_0(\gamma)} + {\vert y_2\vert }^{h_0(\tilde{\gamma})+ \eta} \leq \vert c_{0}({\gamma}) \vert {\vert y_{2}\vert}^{\alpha h_0({\gamma})} + {\vert y_2\vert }^{h_0(\tilde{\gamma})+ \eta},
\end{eqnarray*}
which is a contradiction. Analogously, we have a contradiction for the case $h_0(\tilde{\gamma}) > h_0(\gamma).$ Therefore, we conclude that $h_0(\gamma) = h_0(\tilde{\gamma}).$

It comes from Inequality (\ref{Eq:final}) that  $\dfrac{\vert y_{1} \vert}{ {\vert y_{2} \vert}} $ and $\dfrac{\vert y_{2} \vert}{ {\vert y_{1} \vert}} $ are bounded; more than that, we see that $(y_2/y_1)^{h_0(\gamma)} \to \lambda\in\C^*$ as $y_2\to 0$. 

Given two polar curves $\gamma,{\gamma}\hspace{0,02 cm} {'} \in \Gamma_{1}(l_{1})$ and the corresponding $\tilde{\gamma},\tilde{{\gamma} \hspace{0,02 cm} {'}} \in \Gamma_{2}(l_{2}).$ Then $H(\gamma(y_1),y_1)= (x_2,y_2)$ and $H({\gamma}\hspace{0,02 cm} {'}(y_1),y_1)= ({\tilde x}_{2},{\tilde{y}}_{2})$ satisfy $ y_2 - {\tilde{y}}_{2} = o(y_2).$
Indeed, it follows from the tangency of $\gamma$ and ${\gamma}\hspace{0,02 cm} {'}$ since $H$ is H\"older.

This shows that
$$ c_{0}({\gamma}\hspace{0,02 cm} {'})y_{1}^{h_0({\gamma}\hspace{0,02 cm} {'})} = c_{0}(\tilde{{\gamma}\hspace{0,02 cm} {'}}){\tilde{y}}_{2}^{h_0({\gamma}\hspace{0,02 cm} {'})}+ o(\tilde {y}_{2}^{h_0({\gamma}\hspace{0,02 cm} {'})})=c_{0}({\gamma}\hspace{0,02 cm} {'}){({y_{2}}/y_{1})}^{h_0({\gamma}\hspace{0,02 cm} {'})}y_{1}^{^{h_0({\gamma}\hspace{0,02 cm} {'})}} + o( {y}_{1}^{h_0({\gamma}\hspace{0,02 cm} {'})}). $$
Considering all polar curves in $\in \Gamma_{1}(l_{1})$ and taking the limit as $y_2 \rightarrow 0$ this gives
$$\{c_0(\gamma)y^{h_0(\gamma)} ; \gamma \in \Gamma_{1}(l_{1})\}/ \mathbb{C^{*}} \subset \{c_0(\tilde{\gamma})y^{h_0(\tilde{\gamma})} ; \tilde{\gamma} \in \Gamma_{2}(l_{2})\}/ \mathbb{C^{*}},$$
where $y_2/y_1$ plays the role of a constant of $\mathbb{C^{*}}.$ 

Finally, the other inclusion
$$ \{c_0(\tilde{\gamma})y^{h_0(\tilde{\gamma})} ; \tilde{\gamma} \in \Gamma_{2}(l_{2})\}/ \mathbb{C^{*}}
\subset
\{c_0(\gamma)y^{h_0(\gamma)} ; \gamma \in \Gamma_{1}(l_{1})\}/ \mathbb{C^{*}}$$ is proved by similar arguments. Therefore, $I(\Gamma_{1}(l_{1}))=  I(\Gamma_{2}(l_{2}))$ and ${\rm Inv}(f_1) = {\rm Inv}(f_2),$ as required.
\end{proof}

\subsection{Cuspidal Neighborhoods of Polar Curves} In this section, we give a proof of Proposition \ref{proposition main}. The proof will be based on a detailed analysis of neighborhoods of the polar curve $\Gamma : \partial f/\partial x = 0.$
Let $f (C^2, 0) \rightarrow (C, 0)$ be the germ of an analytic function with Taylor expansion as in Eq. (\ref{eq1}) and mini-regular in $x.$ Fix an analytic arc $\lambda$ as in Eq. (\ref{eq2}). Write 
\begin{equation}\label{eq7} 
F(X,Y) := f (X+ \lambda(Y),Y) := \sum c_{ij}X^{i}Y^{j/N}=\sum X^{i}f_{i}(Y),
\end{equation}
where $f_{i}(Y)=\sum c_{ij}Y^{j/N}.$ Define $h_{i}=ord(f_{i}(Y)), i=1,2,\ldots, k.$ See that $h_{0}=ord(f_{0}(Y))=ord(\sum c_{0j}Y^{j/N})=ord(f(\gamma(y),y)).$
Thus the arc $\lambda$ is a root of $f$ iff  $h_{0} = 0.$
Suppose that $\lambda$ is not a root of $f.$ Let $\xi := \max\limits_{0<m\leq k}\{\frac{h_{0}-h_{m}}{m}, h_{m}<h_{0}\}.$
Consider all the pairs $(m_{1},h_{m_{1}}),(m_{2},h_{m_{2}}), \ldots, (m_{s},h_{m_{s}}),$ where $0< m_{1}, m_{2}, \ldots, m_{s} \leq k$ such that $\frac{h_{0}-h_{m_{1}}}{m_{1}}=\frac{h_{0}-h_{m_{2}}}{m_{2}}= \cdots = \frac{h_{0}-h_{m_{s}}}{m_{s}}=\xi.$
The polynomial $Q(X,Y)=c_{0h_{0}}Y^{h_{0}}+c_{m_{1}h_{m_{1}}}X^{m_{1}}Y^{h_{m_{1}}}+c_{m_{2}h_{m_{2}}}X^{m_{2}}Y^{h_{m_{2}}}+\cdots +c_{m_{s}h_{m_{s}}}X^{m_{s}}Y^{h_{m_{s}}}$ is weighted homogeneous of weights $(\xi,1)$ and degree $h_{0}.$ So, we have $F(X,Y)=Q(X,Y)+ \cdots$. Consider the polynomial of one complex variable $R(z):=Q(z,1)$. The polynomial $R$ is of degree $m_{s}$. Note also that $R(0)$ gives the leading coefficient of $f (\gamma(y),y)=R(0)y^{h_{0}}+ \cdots$. Suppose now that $\tilde{\gamma}(y)=\gamma(y)+ay^{\xi}+ \cdots$. Then the corresponding polynomial for $\tilde{\gamma}$ equals $R(z + a).$

\begin{lemma}\label{lemma3}
Let $x = \gamma(y)$ be a tangential polar arc. Then for any sufficiently small $\delta > 0$ and any set of constants $K\geq 1, C > 0, A > 0, \gamma$ is contained in $Y_{\delta,i,\alpha}, \ i=0,1,2.$
\end{lemma}
\begin{proof} 
Let $R$ be the polynomial of one complex variable, $z$, associated with the polynomial weighted homogeneous of $f$ in $\gamma$ of weights $(\xi,1)$ and degree $h_{0}.$ Fix any $0 < \delta < \sqrt{\xi} - 1,$ a constant $N > 0,$ and another 
constant $\varepsilon^{'} > 0.$ We require $\varepsilon^{'}$ to be small so that $R(z) - R(0)$ has no other roots in $\vert z \vert \leq 3\varepsilon ^{'}$ but $z = 0.$
We suppose $p_{0} = (x_{0}, y_{0}), \ x_{0} = \gamma(y_{0})$ sufficiently close to the origin, how close we shall determine later. 
Let $X^{'}$ denote the connected component of $\tilde{X}=\{(x,y)\in X_{c}; \vert y-y_{0}\vert \leq N \vert y_{0}\vert^{\frac{1+\delta}{\alpha^{8-2i}}})$ that contains $p_{0}.$ 
We show that $X^{'} \subset \{(x,y); \vert x - \gamma(y)\vert \leq  \varepsilon^{'} \vert y \vert ^{\xi} \}.$
For this we consider a continuous function $s \colon X^{'} \to \mathbb{C}$ given by $s(x,y) = (x - \gamma)/ y^{\xi}.$ Let $ p = (x,y) \in X^{'}$ be such that $\{ \vert x - \gamma(y)\vert \leq 2\varepsilon^{'} \vert y \vert ^{\xi}\}$ and write $ x = \gamma(y)+ sy^{\xi}, \ \vert s \vert \leq 2\varepsilon ^{'}.$ Then
$$ f(x,y) = R(s)y^{h_{0}} + o(y^{h_{0}}),$$ 
where $ h_{0} = h_{0}(\gamma).$ On the other hand,
$$ f(x,y) = f(x_{0},y_{0}) = R(0)y_{0}^{h_{0}} + o(y_{0}^{h_{0}}) =  R(0)y^{h_{0}} + o(y^{h_{0}}). $$
Hence
$$R(s) - R(0) = o(1)$$ 
that is to say, it can be arbitrarily small if we have chosen $y_{0}$ sufficiently close to $0.$
Thus, $s$ is close to a root of $R - R(0).$ But, by assumption on $\varepsilon^{'}$, this root has to be
$0.$ Thus, we have shown that if $\vert s \vert \leq 2\varepsilon^{'}$ then it is as close to $0$ as we wish (if $p_{0}$ is close to the origin), that is, for instance, $\vert s \vert \leq \varepsilon^{'}$. Thus a continuous function s defined on
connected $X^{'}$ does not take values in $\varepsilon^{'} < \vert s \vert \leq 2\varepsilon^{'}.$ This shows $X^{'} \subset
\{ \vert x - \gamma(y)\vert \leq \varepsilon^{'}\vert y \vert ^{\xi} \}.$
Consider
$$ 
\pi_{c}^{'} : X^{'} \to V_{N} = \{ y; \vert y - y_{0} \vert  < N\vert y_{0} \vert^{\frac{1+\delta}{\alpha^{8-2i}}} \}
$$
$\pi_{c}^{'}$ is a finite covering branched at the points of polar arcs of the form
$x = \gamma( y) +  o( y^{\xi}).$ Since $P_{0}$ is a branching point, this covering is at least of degree $2.$
Fix $y_{1}$ such that $\vert y_{1} - y_{0}\vert = \frac{1}{2}N\vert y_{0} \vert ^{\frac{1+\delta}{\alpha^{8-2i}}}$ and $p_{1} = (x_{1}, y_{1}), p_{2} = (x_{2},y_{1})$ two distinct points in $(\pi_{c}^{'})^{-1}( y_{1}).$ Then
\begin{eqnarray*}\label{eq2.9} 
d_{p_{0},ML^{-\tfrac{1}{\alpha }-\tfrac{(1+\delta)}{\alpha^{8-2i}}}\vert P_{0} \vert^{\tfrac{1+\delta}{\alpha^{8-2i}}},K,diam}\left( p_{1},%
p_{2}\right) \geq  \vert y_{1}- y_{0}\vert &=& \frac{1}{2}N\vert y_{0} \vert ^{\frac{1+\delta}{\alpha^{8-2i}}}
\end{eqnarray*} 
and
$$\vert p_{1}- p_{2}\vert = \vert x_{1} - x_{2} \vert \leq 2\vert y_{1}\vert^{\xi} \leq 3\vert y_{0} \vert ^{\xi}.$$
Hence,
\begin{eqnarray*}\label{eq2.10} 
\left(d_{p_{0},ML^{-\tfrac{1}{\alpha }-\tfrac{(1+\delta)}{\alpha^{8-2i}}}\vert P_{0} \vert^{\tfrac{1+\delta}{\alpha^{8-2i}}},K,diam}\left( p_{1},p_{2}\right)\right)^{\tfrac{1}{\alpha^{2-i}}} &\geq &  \left(\frac{1}{2}N\vert y_{0} \vert ^{\tfrac{1+\delta}{\alpha^{8-2i}}}\right)^{\tfrac{1}{\alpha^{2-i}}}\\ &\geq& \left(\frac{1}{2}N\right)^{\tfrac{1}{\alpha^{2-i}}}\vert y_{0} \vert ^{\tfrac{1+\delta}{\alpha^{10-3i}}}
\end{eqnarray*}
and
$$\vert p_{1}- p_{2}\vert ^{\alpha^{2-i}} = \vert x_{1} - x_{2} \vert ^{\alpha^{2-i}} \leq (2\vert y_{1} \vert^{\xi})^{\alpha^{2-i}} \leq (3\vert y_{0} \vert ^{\xi})^{\alpha^{2-i}} = 3^{\alpha^{2-i}}\vert y_{0} \vert ^{\xi\cdot\alpha^{2-i}}.$$
Therefore
\begin{eqnarray*}\label{eq2.11} 
\frac{\left(d_{p_{0},ML^{-\tfrac{1}{\alpha }-\tfrac{(1+\delta)}{\alpha^{8-2i}}}\vert P_{0} \vert^{\tfrac{1+\delta}{\alpha^{8-2i}}},K,diam}\left( p_{1},p_{2}\right)\right)^{\tfrac{1}{\alpha^{2-i}}}}{\vert p_{1}- p_{2}\vert ^{\alpha^{2}}} &\geq & \frac{\left(\frac{1}{2}N\right)^{\tfrac{1}{\alpha^{2-i}}}}{3^{\alpha^{2-i}}}\cdot \frac{\vert y_{0} \vert ^{\tfrac{1+\delta}{\alpha^{10-3i}}}}{\vert y_{0} \vert ^{\xi\cdot\alpha^{2-i}}}\\
&\geq & \frac{\tilde M}{\vert y_{0} \vert ^{\xi\cdot\alpha^{(2-i)}-\tfrac{1+\delta}{\alpha^{10-3i}}}},
\end{eqnarray*}
where $\tilde M=\frac{\left(\frac{1}{2}N\right)^{\tfrac{1}{\alpha^{2-i}}}}{3^{\alpha^{2-i}}}$.

Since $\alpha^{12}>\sqrt{1/\xi}$, we have
$$
\begin{array}{lllll}
\dfrac{1}{\alpha ^{12}}\leq \sqrt{\xi} \;\mbox{and}\;1+\delta \leq \sqrt{\xi}&\Longrightarrow& \dfrac{1+\delta }{\alpha ^{12}}&\leq& \xi\\
&\Longrightarrow& \textstyle\dfrac{1+\delta }{\alpha ^{10-3i}}&\leq&\dfrac{1+\delta }{\alpha ^{10}}= \dfrac{1+\delta }{\alpha ^{12}}\cdot \alpha^{2}\\
& & &\leq& \xi\alpha^{2}\leq \xi\alpha^{2-i},
\end{array}
$$
and thus we obtain
\begin{equation*}
\psi _{0,\alpha}(p_{0},M^{\tfrac{1}{\alpha ^{8-2i}}}\left\vert p_{0}\right\vert ^{%
\tfrac{1+\delta }{\alpha ^{8-2i}}},K)\rightarrow \infty \ \mbox{as}\
p_{0}\rightarrow 0.
\end{equation*}
It follows then that, $P_{0} \in Y_{0,\alpha}(\delta,ML^{-\tfrac{1}{\alpha }-t\frac{1+\delta}{\alpha^{8}}},K,AL^{\alpha +\tfrac{1}{\alpha ^{2}}}).$

This proves the lemma.
\end{proof}

Now we will introduce the Henry-Parunsinski neighborhood. Suppose that $p,q\in $ $X\left( p_{0},\rho \right) \ $ belong to the same connected component of $X(p_{0},K\rho^{\alpha}),$ as $p_{0}.$
Then one can join $p$ and $q$ by a piecewise $C^1$ curve in $X(p_{0},K\rho^{\alpha}).$ Let $d_{p_{0},\rho,K}(p,q)$
denote the infimum of lengths of such curves that is the intrinsic distance of $p$ and $q$ in $X(p_{0},K\rho).$ Define
\begin{equation*}
\varphi \left( p_{0},\rho ,K\right) :=\sup \dfrac{
d_{p_{0},\rho ,K}\left( p,q\right)}{\left\vert
p-q\right\vert },
\end{equation*}%
where the supremum is taken over all pairs of points $p,q$ of $X(p_{0},\rho)$ from the connected component of $X(p_{0},K \rho)$ containing $p_{0}.$ Clearly, if $p_{0}$ is a nonsingular point of $f^{-1}(c)$ then $\varphi(p_{0},\rho,K) \to 1$ as $\rho \to 0$ but $\varphi$ is not necessarily an increasing function of $\rho$ so we define
\begin{equation*}
\psi \left( p_{0},\rho ,K\right) :=\sup_{\rho'\leq \rho
}\varphi \left( p_{0},\rho',K\right) .
\end{equation*}%
Finally, we define%
\begin{equation*}
Y(\rho ,K,A):=\{p;\psi \left( p_{0},\rho ,K\right) \geqslant A\}.
\end{equation*}%

\begin{lemma}\label{lemma4}
If $K$ and $A$ are sufficiently large, then $Y_{2,\alpha}(\rho, K, A) \subset Y(\rho, K, A).$
\end{lemma}
\begin{proof}
Let $p_1,p_2\in X(p_{0},\rho)$ be in the connected component of $X(p_{0},K\rho)$ containing $p_{0}.$
Let\ $\gamma \colon\left[ a.b\right] \rightarrow
X(p_{0},K\rho)$ with $\gamma \left( a\right) =p_{1}\mbox{ and }\gamma \left( b\right) =p_{2}$ joining $p_{1}\mbox{ and } p_{2}.$ 
Hence,
\begin{eqnarray*}
d_{p_{0},\rho ,K,diam}\left( p_{1},%
p_{2}\right) &=&\inf_{\mu }diam\left( \mu \right) \leq
diam\left( \gamma \right) =\sup \left\vert \gamma \left( t_{i}\right) -\gamma
\left( t_{j}\right) \right\vert \\
&\leq &\sup \sum \left\vert  \gamma \left( t_{i-1}\right) - 
\gamma \left( t_{i}\right)  \right\vert = l(\gamma) , 
\end{eqnarray*}%
where $l(\gamma)$ denotes the length of $\gamma.$
So, 
\begin{eqnarray}\label{eq2.8} 
d_{p_{0},\rho ,K,diam}\left( p_{1},%
p_{2}\right) \leq  d_{p_{0},\rho
,K}\left( p_{1},p_{2}\right) .
\end{eqnarray}
Thus,
$$\dfrac{d_{p_{0},\rho,K,diam}\left( p_{1},p_{2}\right) }{%
\left\vert p_{1}-p_{2}\right\vert} \leq
\dfrac{ d_{p_{0},\rho
,K}\left(p_{1},p_{2}\right)}{\left\vert
p_{1}-p_{2}\right\vert}.$$
Then, 
\begin{equation*}
\psi _{2,\alpha}(p_{0},\rho ,K) \leq  \psi\left( p_{0},\rho,K\right) .
\end{equation*}%
Therefore,
\begin{equation*}
Y_{2,\alpha}(\rho, K, A) \subset Y(\rho, K, A).
\end{equation*}%
\end{proof}

A similar argument shows the following.

\begin{lemma}\label{lemma5} If $\delta > 0$ and $Y_{\delta} := Y(\delta,M,K,A):=\{p;\psi(p,M\left\vert p\right\vert ^{%
1+\delta },K)\geqslant A\},$ then $Y_{\delta, i,\alpha} \subset Y_{\delta}, \ i=0,1,2.$
\end{lemma}

Henry and Parusi\'nski, in \cite[Proposition 4.5]{HenryP:2003}, show that the union $\tilde{\Gamma}$ of the tangent arcs is contained in $Y_{\delta}$ and furthermore for each $p_0 \in Y_{\delta}$ there is $p \in \tilde{\Gamma}$ such that $f(p) = f(p_0)$ and
\begin{equation}\label{eq12}
| p - p_{0} | \leq B {| p_{0} |}^{1+ \delta}.
\end{equation}

\begin{proof}[Proof of Proposition \ref{proposition main}]

For the proof of the first part, we use Lemma \ref{lemma3}. For the second part, we use Lemma \ref{lemma5} and Inequality (\ref{eq12}).
\end{proof}

\subsection{An application: Existence of continuous moduli} \label{subsec:existence_moduli}
The aim of the next example is to answer Question \ref{main_question}. Actually, we are going to show that: given the $1$-parameter family of weighted homogeneous  polynomials
$f_{t}\colon\left(\mathbb{C}^{2},0\right) \rightarrow
\left(\mathbb{C},0\right) $ , $t\in\mathbb{C},$ given by
\begin{equation}\label{eq1.1}
f_{t}\left( x,y\right) =x^{3}-3t^{2}xy^{2d}+y^{3d}
\end{equation}
where $d>1$ is an integer number, there exists an uncountable subset $I\subset \C$ such that $f_t$ is not H\"older equivalent to $f_s$ for any $t\neq s\in I$.

Our strategy to compute the number ${\rm Inv}(f_t)$ of the family  of function-germs $(\mathbb{C}^2,0)\rightarrow(\mathbb{C},0)$
\begin{equation*}
f_{t}\left( x,y\right) =x^{3}-3t^{2}xy^{2d}+y^{3d}.
\end{equation*}

Indeed, the polar arcs of $f_t$ are given by the following equations:
$$0=\frac{\partial f_t}{\partial x} = 3x^2-3t^2y^{2d},$$ which means that they are parametrized by $x=\gamma_1(y)=ty^{d}$ and $x=\gamma_2(y)=-ty^{d}$. Then

\begin{eqnarray*}
f_t(\gamma_1(y),y) &=& (ty^d)^3-3t^2(ty^d)^{2d}+y^{3d} \\
&=& (1-2t^3)y^{3d}
\end{eqnarray*}

and

\begin{eqnarray*}
f_t(\gamma_2(y),y) &=& (-ty^d)^3-3t^2(-ty^d)^{2d}+y^{3d} \\
&=& (1+2t^3)y^{3d},
\end{eqnarray*}
hence ${\rm Inv}(f_t) = \{ (1-2t^3)y^{3d}, (1+2t^3)y^{3d} \} / \C^*$. Finally, by Theorem \ref{main theorem}, if $f_t$ is H\"older equivalent to $f_s$, then ${\rm Inv}(f_t)={\rm Inv}(f_s)$ and, it follows that  $t=\pm \xi s$, where $\xi\in\C$ is a primitive cubic unit root.

\end{document}